\documentclass{amsart}
\usepackage[utf8]{inputenc}
\usepackage[utf8]{inputenc}
\usepackage[english]{babel}
\usepackage{amssymb}
\usepackage{amsfonts}
\usepackage{amsmath}
\usepackage{amscd}
\usepackage{amsthm}
\usepackage{geometry} 
\usepackage{latexsym}
\usepackage{epsfig}
\usepackage{graphicx}
\usepackage{mathtools}
\usepackage{xcolor}
\usepackage[all]{xypic}

\newtheorem{theorem}{Theorem} [section]
\newtheorem{corollary}[theorem]{Corollary}
\newtheorem{definition}[theorem]{Definition}
\newtheorem{proposition}[theorem]{Proposition}
\newtheorem{lemma}[theorem]{Lemma}
\newtheorem{remark}[theorem]{Remark}

\title[Geometric aspects on Humbert-Edge's curves of type 5]{Geometric aspects on Humbert-Edge's curves of type 5, Kummer surfaces and hyperelliptic curves of genus 2}
\author{Abel Castorena and Juan Bosco Fr\'ias-Medina}
\date{}

\address{Centro de Ciencias Matem\'aticas. Universidad Nacional Aut\'onoma de M\'exico, Campus Morelia. Antigua Carretera a P\'atzcuaro 8701, Col. Ex-Hacienda San Jos\'e de la Huerta, C.P. 58089. Morelia, Michoac\'an, Mexico}
\email{abel@matmor.unam.mx}
\email{bosco@matmor.unam.mx}

\thanks{The first author is supported by Grants PAPIIT UNAM IN100419 ``Aspectos Geom\'etricos del moduli de curvas $M_g$" and CONACyT, M\'exico A1-S-9029 ``Moduli de curvas y curvatura en $A_g$". The second author is supported by Programa de Becas Posdoctorales 2019, DGAPA, UNAM}

\begin{document}

\begin{abstract}
In this work we study the Humbert-Edge's curves of type 5, defined as a complete intersection of four diagonal quadrics in $\mathbb{P}^5$. We characterize them using Kummer surfaces and using the geometry of these surfaces we construct some vanishing thetanulls on such curves. In addition, we describe an argument to give an isomorphism between the moduli space of Humbert-Edge's curves of type 5 and the moduli space of hyperelliptic curves of genus 2, and we let see how this argument can be generalized to state an isomorphism between the moduli space of hyperelliptic curves of genus $g=\frac{n-1}{2}$ and the moduli space of Humbert-Edge's curves of type $n\geq 5$ where $n$ is an odd number.
\end{abstract}

\maketitle


\section{Introduction}
W.L. Edge began the study of \emph{Humbert's curves} in \cite{Edg51}, such curves are defined as canonical curves in $\mathbb{P}^4$ that are complete intersection of three diagonal quadrics. A natural generalization of Humbert's curve was later introduced by Edge in \cite{Edg77}: irreducible, non-degenerated and non-singular curves on $\mathbb{P}^n$ that are the complete intersection of $n-1$ diagonal quadrics. One important feature of Humbert-Edge's curves of type $n$ noted by Edge in \cite{Edg77} is that each one can be write down in a \emph{normal form}. Indeed, we can assume that the $n-1$ quadrics in $\mathbb{P}^n$ are given by 
\begin{equation*}
    Q_i=\sum_{j=0}^n a_j^{i} x_j^2, \;\; i=0,\dots,n-1
\end{equation*}
where $a_j\in\mathbb{C}$ for all $j\in\{0,\dots,n\}$ and $a_j\neq a_k$ if $j\neq k$. We say that a curve satisfying this conditions is a \emph{Humbert-Edge's curve of type $n$}. Note that in the case of a Humbert's curve $X$, i.e. when $n=4$, this form of the equations implies directly that $X$ is contained in a degree four Del Pezzo surface.  

The Humbert-Edge's curve of type $n$ for $n>4$ have been studied in just a few works. Carocca, Gonz\'alez-Aguilera, Hidalgo and Rodr\'iguez studied in \cite{CGHR08} the Humbert-Edge's curves from the point of view of uniformization and Klenian groups. Using a suitable form of the quadrics, Hidalgo presented in \cite{Hid09} and \cite{Hid12} a family of Humbert-Edge's curves of type 5 whose fields of moduli are contained in $\mathbb{R}$ but none of their fields of definition are contained in $\mathbb{R}$.
Fr\'ias-Medina and Zamora presented in \cite{FZ18} a characterization of Humbert-Edge's curves using certain abelian groups of order $2^n$ and presented specializations admitting larger automorphism subgroups. Carvacho, Hidalgo and Quispe determined in \cite{CHQ16} the decomposition of the Jacobian of generalized Fermat curves, and as a consequence for Humbert-Edge's curves.  Auffarth, Lucchini Arteche and Rojas described in \cite{ALR21} the decomposition of the Jacobian of a Humbert-Edge's curve more precisely given the exact number of factors in the decomposition and their corresponding dimensions.

In the case $n=5$, the normal form for these curves implies that they are contained in a special $K3$ surface, a \emph{Kummer surface}. In this work,  we study the Humbert-Edge's curve of type 5 and determine some properties using the geometry of Kummer surfaces.

Recall that an \emph{algebraic (complex) $K3$ surface} is a complete non-singular projective (compact connected complex) surface $S$ such that $\omega_S\cong \mathcal{O}_S$ and $H^1(S,\mathcal{O}_S)=0$. Classically, a \emph{(singular) Kummer surface} is a surface in $\mathbb{P}^3$ of degree 4 with 16 nodes and no other singularities. An important fact about Kummer surfaces is that are determined by the set of their nodes. One can take the resolution of singularities of a Kummer surface, the obtained non-singular surface is a $K3$ surface. For our purposes, these non-singular models will be called Kummer surfaces.

The Kummer surfaces that we will take for our study are those coming from a hyperelliptic curve of genus 2. In such case, an important feature shared by Humbert-Edge's curves of type 5 and Kummer surfaces is that they admit the automorphism subgroup generated by the natural involutions of $\mathbb{P}^5$ that change the $i$-th coordinate by its negative:
\begin{eqnarray*}
 \sigma_i :\hspace{3mm}  \mathbb{P}^5 \hspace{12mm} & \rightarrow & \hspace{14mm} \mathbb{P}^5 \\
 (x_0:\cdots:x_i\cdots:x_5) &\mapsto & (x_0:\cdots:-x_i:\cdots:x_5)
\end{eqnarray*}
These automorphisms along with Knutsen's result \cite{Knu02} on the existence of a $K3$ surface of degree $2n$ in $\mathbb{P}^{n+1}$ containing a smooth curve of genus $g$ and degree $d$ will enable us to characterize the Humbert-Edge's curves of type 5 using the geometry of a Kummer surface. 

This work is organized as follows. In Section \ref{kummer} we review the construction of a Kummer surface from a two-dimensional torus and conditions that ensure when a Kummer surface is projective. Later, we focus on the case of Kummer surface obtained from hyperelliptic curves of genus 2. Section \ref{hetype5} is splited into three parts. First, in Section \ref{hekummer}, we review the basic properties of Humbert-Edge's curves of type 5 and characterize them using the geometry of the Kummer surface. Later, in Section \ref{theta}, we present the construction of some odd theta characteristic on a Humbert-Edge's curve of type 5 using the automorphisms $\sigma_i$'s and some vanishing thetanulls using the Rosenhain tetrahedra associated to the Kummer surface. Finally, in Section \ref{moduli} we use the embedding given in \cite{CGHR08} to construct an isomorphism between the moduli space $\mathcal{HE}_5$ of Humbert-Edge's curve of type 5 and the moduli space $\mathcal{H}_2$ of hyperelliptic curve of genus 2 and as a consequence, we obtain that $\mathcal{H}_2$ is a three-dimensional closed subavariety of $\mathcal{M}_{17}$. Moreover, we generalized this argument to show that there is an isomorphism between the moduli space $\mathcal{HE}_n$ of Humbert-Edge's curves of type $n$, where $n\geq 5$ is an odd number, and the moduli space $\mathcal{H}_{g}$ of hyperelliptic curves of genus $g=\frac{n-1}{2}$.

\section{Kummer surfaces}\label{kummer}

\subsection{Construction from a two-dimensional torus}
In this paper the ground field is the complex numbers. In this section, we recall the construction of the Kummer surface associated with a two-dimensional torus.

Let $T$ be a two-dimensional torus. Consider the involution $\iota:T\rightarrow T$ which sends $a\mapsto -a$ and take the quotient surface $T/\langle \iota \rangle$. The surface $T/\langle \iota \rangle$ is known as the \emph{singular Kummer surface} of $T$. It is well-known that this surface has 16 ordinary singularities and resolving them we obtain a $K3$ surface called the \emph{Kummer surface} of $T$ and denoted by $\mathrm{Km}(T)$ (see for example \cite[Theorem 3.4]{Gon94}). This procedure is called the \emph{Kummer process}. Note that by construction, $\mathrm{Km}(T)$ has 16 disjoint smooth rational curves, indeed, they correspond to the singular points of the quotient surface. Conversely, Nikulin proved in \cite{Nik75} the converse:  

\begin{theorem}
If a $K3$ surface $S$ contains 16 disjoint smooth rational curves, then there exists a unique complex torus, up to isomorphism, such that $S$ and the rational curves are obtained by the Kummer process. In particular, $S$ is a Kummer surface. 
\end{theorem}

Note that the above construction holds true for any two-dimensional torus, not a necessary projective one. In particular, with this process it is possible construct $K3$ surfaces that are not projective. However, there is an equivalence between the projectivity of the torus and the associated $K3$ surface (see \cite[Theorem 4.5.4]{BL04}):

\begin{theorem}
 Let $T$ be a two-dimensional torus. $T$ is an abelian surface if an only if $\mathrm{Km}(T)$ is projective. 
\end{theorem}

Now, if $A$ is a principally polarized abelian surface, then $A$ is one of the following (see \cite[Corollary 11.8.2]{BL04}):
\begin{enumerate}
 \item[a)] The Jacobian of a smooth hyperelliptic curve of genus 2, or
 \item[b)] The canonical polarized product of two elliptic curves.
\end{enumerate}
As we will see next, the Case a) is the one of our interest.

\subsection{Hyperelliptic curves of genus 2 and Kummer surfaces}\label{hyperkummer}
We are interested in $K3$ surfaces that are a smooth complete intersection of type $(2,2,2)$ in $\mathbb{P}^5$, i.e. that are a complete intersection of three quadrics. Moreover, we restrict to the case in which the quadrics are diagonal. The interest of having diagonal quadrics defining the $K3$ surface is that they enable us to work with hyperelliptic curves of genus 2. 

Indeed, let $C$ be the hyperelliptic curve of genus 2 given by the affine equation
\begin{equation}\label{hypcurve}
  y^2=f(x)=(x-a_0)(x-a_1)\cdots(x-a_5),
\end{equation}
where $a_0,\dots,a_5\in\mathbb{C}$ and $a_i\neq a_j$ if $i\neq j$. We can consider the jacobian surface $J(C)$ associated with $C$ and applying the Kummer process we obtain that the $K3$ surface $\mathrm{Km}(J(C))$ is isomorphic to the surface in $\mathbb{P}^5$ defined by the complete intersection of the 3 diagonal quadrics by \cite[Theorem 2.5]{Shio77}.
\begin{equation}\label{normalkummer}
 Q_i=\sum_{j=0}^5 a_j^i x_j^2, \;i=0,1,2.
\end{equation}

In order to obtain a hyperelliptic curve of genus 2 beginning with a smooth $K3$ surface $S$ in $\mathbb{P}^5$ given by the complete intersection of three diagonal quadrics, we may assume an additional hypothesis. Edge studied in \cite{Edg67} the Kummer surfaces defined by \eqref{normalkummer}. One of his result establishes that whenever a surface $X$ given by the intersection of three linearly independent quadrics has a common self-polar simplex $\Sigma$ in $\mathbb{P}^5$ and contains a line in general position, then the equations defining $X$ can be written with the form \eqref{normalkummer}. Observe that this fact is equivalent to require that $X$ contains 16 disjoint lines, indeed, using the natural involutions of $\mathbb{P}^5$ one can obtain the others lines.

Then, let $S$ be a smooth $K3$ surface in $\mathbb{P}^5$ given by the complete intersection of the quadrics
\begin{equation*}
 Q_i=\sum_{j=0}^5 a_{ij} x_j^2, \;\;\; i=0, 1, 2,
\end{equation*}
where $a_{ij}\in\mathbb{C}$ for $i=0,1,2$ and $j=0,\dots,5$ and assume that $S$ contains 16 disjoint lines. As a consequence, we may assume that $S$ is given by the quadrics in \eqref{normalkummer}
for some $a_i\in\mathbb{C}$ where $a_i\neq a_j$ if $i\neq j$. By \cite[Theorem 1]{Nik75} there exists an unique (up to isomorphism) two-dimensional torus that gives rise the surface $S$. Taking the hyperelliptic curve $C$ given by the equation $y^2=f(x)=(x-a_0)(x-a_1)\cdots(x-a_5)$, we obtain that $S$ is isomorphic to $\mathrm{Km}(J(C))$.

From now on, we say that a Kummer surface is a smooth surface in $\mathbb{P}^5$ given by the complete intersection of three diagonal quadrics as in \eqref{normalkummer}. 

For a Kummer surface $S$ given by \eqref{normalkummer}, it is possible give the parametric form of the 32 lines contained in $S$. Indeed, in \cite{Edg67} Edge noted that the equation of a line $\ell$ contained in $S$ is given in the following parametric form:
\begin{equation}\label{line_kummer}
 \left(\frac{t+a_0}{\sqrt{f'(a_0)}}, \frac{t+a_1}{\sqrt{f'(a_1)}},\frac{t+a_2}{\sqrt{f'(a_2)}},\frac{t+a_3}{\sqrt{f'(a_3)}},
 \frac{t+a_4}{\sqrt{f'(a_4)}},\frac{t+a_5}{\sqrt{f'(a_5)}}\right).
\end{equation}
Recall that the natural automorphisms $\sigma_i: x_i\mapsto -x_i$ of $\mathbb{P}^5$ act on $S$. Denote by $E$ the group $\langle \sigma_0,\dots,\sigma_5 \rangle$. Applying each element of $E$ to the line $\ell$, we obtain the other 32 lines on $S$. The identity gives the line $\ell$ and the remaining elements give the other 31 lines:
\begin{itemize}
 \item $\ell_i:=\sigma_i(\ell)$ for all $i\in\{0,\dots,5\}$,
 \item $\ell_{ij}:=\sigma_i\sigma_j(\ell)$ for different $i,j\in\{0,\dots,5\}$
 \item $\ell_{ijk}:=\sigma_i\sigma_j\sigma_k(\ell)$ for different $i,j,k\in\{0,\dots,5\}$.
\end{itemize}

It is well-know that a singular Kummer surface $K$ is birational to a Weddle surface $W$ (see for example \cite[Proposition 1]{Var86}). A \emph{Weddle surface} is a quartic surface in $\mathbb{P}^3$ with 6 nodes. The 32 lines on $S$ have a geometric interpretation in both $K$ and $W$ as Edge pointed in \cite{Edg67}. Indeed, the projection $\pi$ of $S$ from $\ell$ is a Weddle surface $W$ and it occurs:
\begin{itemize}
 \item $\pi(\ell_i)=k_i$ is a node on $W$, for all $i=0,\dots,5$,
 \item $\pi(\ell_{ij})$ is the line through $k_i$ and $k_j$, for different indices $i,j\in\{0,\dots,5\}$, 
 \item $\pi(\ell_{ijh})$ is the line in the intersection of the plane generated by $k_i,k_j,k_h$ with the complementary plane, for different indices $i,j,h\in\{0,\dots,5\}$, and 
 \item $\pi(\ell)$ is the cubic on $W$ through the six nodes $k_i$'s.
\end{itemize}
On the other hand, since $S$ is the resolution of singularities of $K$ it occurs:
\begin{itemize}
 \item The 16 nodes of $K$ correspond to the 16 lines $\ell_{i}$ and $\ell_{ijk}$, and
 \item The conics of contact of $K$ with its 16 tropes correspond to the 16 lines $\ell$ and $\ell_{ij}$. 
\end{itemize}

A \emph{trope} is a plane which intersects the quartic along a conic. The nodes and the tropes of a singular Kummer surface provide an interesting configuration on it.

\begin{definition}
Let $\Gamma$ be a set of $16$ planes and $16$ points in $\mathbb{P}^3$.
\begin{itemize}
    \item $\Gamma$ is a \emph{$(16,6)$-configuration} if every plane contains exactly $6$ of the $16$ points and every point lies in exactly $6$ of the $16$ planes. The 16 planes are called \emph{special planes}.  
    \item A $(16,6)$-configuration is \emph{non-degenerated} if every two special planes shared exactly two points of the configuration and every pair of points is contained in exactly two special planes.
    \item An \emph{abstract $(16,6)$-configuration} is a $16\times 16$ matrix $(a_{ij})$ whose entries are ones or zeros, with exactly $6$ ones in each row and in each column. The rows of the matrix are called \emph{points} of the configuration and the columns are called \emph{planes} of the configuration. The $i$-th point \emph{belongs} to the $j$-th plane if and only if $a_{ij}=1$.
\end{itemize}
 
\end{definition}

Gonzalez-Dorrego classified in \cite{Gon94} the non-degenerated $(16,6)$-configurations and used them to classify the singular Kummer surfaces. Given a singular Kummer surface, the nodes and the tropes establish a non-degenerated $(16,6)$-configuration (see \cite[Corollary 2.18]{Gon94}) and conversely, given a $(16,6)$-configuration, there exists a singular Kummer surface whose associated $(16,6)$-configuration is the given one (see \cite[Theorem 2.20]{Gon94}). 

\begin{definition}
A \emph{Rosenhain tetrahedron} in an abstract $(16,6)$-configuration is a set of $4$ points and $4$ planes such that each plane contains exactly $3$ points and each point belongs to exactly $3$ planes. The $4$ points are the \emph{vertices} of the tetrahedron. An \emph{edge} is a pair of vertices and a \emph{face} a triple of vertices.
\end{definition}

Rosenhain tetrahedra always exist in a singular Kummer surface, in fact, there exist $80$ of them (\cite[Corollary 3.21]{Gon94}). Moreover, these tetrahedra are relevant because using them we can construct divisors that are linearly equivalent and whose class induces the closed embedding to $\mathbb{P}^5$ (see \cite[Proposition 3.22 and Remark 3.24]{Gon94}):

\begin{proposition}\label{rosenhain}
Given a Rosenhain tetrahedron on a singular Kummer surface $K$, let $D$ be the divisor on the associated Kummer surface $S$ given by the sum of proper transforms of the $4$ conics in which the planes meet on $K$ and the $4$ exceptional divisors corresponding to the $4$ nodes. Then, the linear equivalence class of $D$ is independent of the choice of the Rosenhain tetrahedron. In addition, $D^2=8$, $\mathrm{dim}|D|=5$ and the linear system $|D|$ induces a closed embedding of $S$ in $\mathbb{P}^5$ as the complete intersection of three quadrics.
\end{proposition}
 
These divisors will be used in next section to construct vanishing thetanulls on Humbert-Edge's curves of type 5.

\section{Humbert-Edge curves of type 5}\label{hetype5}

\subsection{Properties and characterization}\label{hekummer}

Here, we review the main properties of the Humbert-Edge's curve of type 5 and present a characterization using the lines lying on a Kummer surface.

\begin{definition}
 An irreducible, non-degenerated and non-singular curve $X_5\subseteq\mathbb{P}^5$ is a \textit{Humbert-Edge's curve of type 5} if it is the complete intersection of 4 diagonal quadrics $Q_0,\dots,Q_3$:
\begin{equation*}
 Q_i=\sum_{j=0}^5 a_{ij} x_j^2, \;\;\; i=0,\dots, 3. 
\end{equation*}
\end{definition}

The basic properties of a Humbert-Edge's curve of type 5 are stated below.

\begin{lemma}
 Let $X_5\subset\mathbb{P}^5$ be a Humbert-Edge's curve of type 5. The following hold:
\begin{enumerate}
  \item $X_5$ is a curve of degree 16.
  \item The genus of $X_5$ is equal to $g(X_5)=17$.
  \item Every 4-minor of the matrix $(a_{ij})$ is non-degenerated.
  \item $X_5$ is non-trigonal.
\end{enumerate}
\end{lemma}

The diagonal form of the equations defining a Humbert-Edge's curve $X_5$ of type 5 implies that it admits the action of the group $E$ generated by the six involutions $\sigma_i: x_i\mapsto -x_i$ acting with fixed points and whose product is the identity. 
Moreover, these involutions establish a relation between the Humbert-Edge's curves of type 5 and the Humbert's curves in $\mathbb{P}^4$. For every $i=0,\dots,5$, we can consider the covering $\pi_i:X_5\rightarrow X_5/\langle \sigma_i\rangle$ induced by the involution $\sigma_i$. This is a two-to-one covering ramified at 16 points obtained as the intersection points of $X_5$ with the hyperplane $V(x_i)$. In addition, the quotient of $X_5/\langle \sigma_i\rangle$ is a Humbert's curve in $\mathbb{P}^4$. This double covering can be interpreted geometrically as the projection of $X_5$ with center $e_i$ onto the hyperplane $V(x_i)$.

Next result shows that a Humbert-Edge's curve of type 5 is always contained in a Kummer surface. It is a consequence of the fact noted by Edge in \cite{Edg77} that a Humbert-Edge's curve of type $n$ can be written in a normal form.

\begin{proposition}\label{he5kummer}
Let $X_5\subset\mathbb{P}^5$ be a Humbert-Edge's curve of type 5. There exists a Kummer surface in $\mathbb{P}^5$ which contains $X_5$.
\end{proposition}
\begin{proof}
Assume that $X_5$ is given by the equations
\begin{equation*}
 Q_i=\sum_{j=0}^5 a_{ij} x_j^2, \;\;\; i=0,\dots,3
\end{equation*}
where $a_{ij}\in\mathbb{C}$. For each $j=0,\dots,5$, consider the coefficients $a_{ij}$ as the entries of the point $p_j=(a_{0j}:a_{1j}:a_{2j}:a_{3j})$ in the projective space $\mathbb{P}^3$. We have six points $p_0,\dots,p_5$ in $\mathbb{P}^3$ that are in general position, so there exists a unique rational normal curve $C\subset\mathbb{P}^3$ through these points. Finally, take a change of coordinates of $\mathbb{P}^3$ such that $C$ is in the standard parametric form, then we may assume that $p_j=(1:a_j:a_j^2:a_j^3)$ for all $j=0,\dots,5$. Therefore, we obtain that $X_5$ is given by the equations
\begin{equation}\label{normalx5}
 Q_i=\sum_{j=0}^5 a_{j}^i x_j^2, \;\;\; i=0,\dots, 3
\end{equation}
with $a_j\in\mathbb{C}, j=0,\dots,5$ and $a_i\neq a_i\hskip1mm\forall i\neq j$. The quadrics $Q_0,Q_1,Q_2$ define the Kummer surface associated with the hyperelliptic curve $y^2=\prod_{j=0}^5 (x-a_j)$. Therefore, $X_5$ is contained in a Kummer surface.
\end{proof}

\begin{remark}\label{x5hyp2}
 Note that given a Humbert-Edge's curve $X_5$ of type 5 in normal form \eqref{normalx5}, by the above proposition it is always possible to find an hyperelliptic curve $C$ such that the Kummer surface $\mathrm{Km}(J(C))$ contains $X_5$. Reciprocally, given an hyperelliptic curve $C$ of genus $2$, by the discussion in Section \ref{hyperkummer}, in a natural way the associated Kummer surface $\mathrm{Km}(J(C))$ contains a Humbert-Edge's curve of type 5 whose equations are in the normal form \eqref{normalx5}. Denote by $\mathcal M_g$ the moduli space of smooth and irreducible curves of genus $g$. Note that Proposition \ref{he5kummer} let us see that a Humbert-Edge's curve of type 5 depends of three parameters in $\mathcal M_{17}$, in fact in Section \ref{moduli} we prove that the moduli space of Humbert-Edge's curve of type 5 is isomorphic to the moduli space of hyperelliptic curves of genus 2.
\end{remark}

Next we present a characterization for Humbert-Edge's curves of type 5 using the lines on Kummer surfaces. 

\begin{theorem}
 Let $X\subset\mathbb{P}^5$ be an irreducible, non-degerated and non-singular curve of degree 16 and genus 17. The following statements are equivalent: 
 \begin{enumerate}
  \item[(i)] $X$ is a Humbert-Edge's curve of type 5.
  \item[(ii)] $X$ admits six involutions $\sigma_0,\dots,\sigma_5$ such that $\langle \sigma_0,\dots,\sigma_5\rangle\cong(\mathbb{Z}/2\mathbb{Z})^5$, $\sigma_0\cdots\sigma_5=1$ and the quotient $X/\langle\sigma_i\rangle$ is a Humbert's curve for every $i=0,\dots,5$.
  \item[(iii)] There exists a Kummer surface $S$ which contains $X$ and such that the intersection of $X$ with the 16 lines $\ell$ and $\ell_{ij}$ is at most one point.
  \item[(iv)] There exists a Kummer surface $S$ which contains $X$ and such that the intersection of $X$ with the 16 lines $\ell_{i}$ and $\ell_{ijk}$ is at most one point.
 \end{enumerate}
\end{theorem}
\begin{proof} (i)$\Leftrightarrow$(ii) We have this equivalence by \cite[Theorem 3.4]{FZ18}. 

(i)$\Rightarrow$(iii)  Assume that $X$ is Humbert-Edge's curve of type 5. Proposition \ref{he5kummer} implies that $X$ is contained in a Kummer surface $S$. So, we may assume the existence of different scalars $a_0,\dots,a_5\in\mathbb{C}$ such that $X$ is given by the equations
\begin{equation*}
 Q_i=\sum_{j=0}^5 a_{j}^i x_j^2, \;\;\;i=0,\dots,3
\end{equation*}
and $S$ is given by the equations $Q_0,Q_1,Q_2$. Consider the line $\ell$ in parametric form as in $\eqref{line_kummer}$. A direct computation shows that for every $t\in\mathbb{C}$,
\begin{equation*}
 Q_3\left(\frac{t+a_0}{\sqrt{f'(a_0)}}, \frac{t+a_1}{\sqrt{f'(a_1)}},\frac{t+a_2}{\sqrt{f'(a_2)}},\frac{t+a_3}{\sqrt{f'(a_3)}},
 \frac{t+a_4}{\sqrt{f'(a_4)}},\frac{t+a_5}{\sqrt{f'(a_5)}}\right)=1.
\end{equation*}
Therefore, $X$ does not intersect the line $\ell$ and the diagonal form of the third equation implies that $X$ also does not intersect the lines $\ell_{ij}$ for every $i,j\in\{0,\dots,5\}$ with $i\neq j$.

(iii)$\Rightarrow$(i) Assume that $X$ is contained in a Kummer surface $S$ defined by the equations
\begin{equation*}
 Q_i= \sum_{j=0}^5 a_j^i x_j^2, \;\;\; i=0,1, 2,
\end{equation*}
where $a_i\neq a_j$ if $i\neq j$, and such that $X$ does not intersect the lines $\ell$, $\ell_{ij}$ in two different points for all $i,j\in\{0,\dots,5\}$ with $i\neq j$. We denote $f(x)=\prod_{i=0}^5 (x-a_i)$. Since $S$ is a $K3$ surface of type $(2,2,2)$ in $\mathbb{P}^5$ containing $X$, our situation should be one of the cases determined by Knutsen in \cite[Theorem 6.1 (3)]{Knu02}. In fact, we are in the case a) of the latter, in Knutsen's notation we have $n=4$, $d=16$, $g=17$ and $g=d^2/16 +1$. Moreover, such result implies that $X$ is the complete intersection of $S$ and a hypersurface of degree $d/8$, i.e. $X$ is the complete intersection of four quadrics. Denote the fourth quadric by
\begin{equation*}
 Q=\sum_{j=0}^5 d_j x_j^2 + \sum_{0\leq i < j\leq 5} d_{ij} x_ix_j.
\end{equation*}
Now, when we evaluate the quadric $Q$ in the parametric form of $\ell$ we obtain a quadratic equation with parameter $t$ with leading coefficient 
\begin{equation*}
 \frac{1}{\prod_{i=0}^5 f'(a_i)} \left( \sum_{i=0}^5 f'(a_0)\cdots \widehat{f'(a_i)} \cdots f'(a_5) d_i 
+ \sum_{0\leq i<j\leq 5}  f'(a_0)\cdots \sqrt{f'(a_i)}\sqrt{f'(a_j)} \cdots f'(a_5) d_{ij} \right).
\end{equation*}
The hypothesis that $Q$ does not intersect the line $\ell$ in two different points implies that such coefficient vanishes (the coefficient of $t$ could vanishes but in such case the constant term must be different from zero). This also occurs for all of the 15 lines $\ell_{ij}$ by hypothesis, and then we have 16 imposed conditions. The leading coefficient for the line $\ell_{ij}$ can be deduced from the above one, in fact, since the line $\ell_{ij}$ is obtained from $\ell$ by the application of $\sigma_i\sigma_j$, it is enough to add a negative sign to the coefficient of the terms $d_{rs}$ whenever $r$ or $s$ are equal to $i$ or $j$. Solving the linear system in the variables $d_j$'s and $d_{ij}$'s, we obtain that all the $d_{ij}$'s are equal to zero, that $d_0,d_1,d_2,d_3$ and $d_4$ are free parameters and
\begin{equation}\label{d5diagonal}
 d_5= - f'(a_5)\left(\frac{d_0}{f'(0)}+\frac{d_1}{f'(a_1)}+\frac{d_2}{f'(a_2)}+\frac{d_3}{f'(a_3)}+\frac{d_4}{f'(a_4)} \right).
\end{equation}
Therefore, $X$ is the complete intersection of four diagonal quadrics in $\mathbb{P}^5$ and we conclude that it is a Humbert-Edge's curve of type 5.

(iii)$\Leftrightarrow$(iv) As above, assuming that $X$ is contained in a Kummer surface $S$ defined by the equations
\begin{equation*}
 Q_i= \sum_{j=0}^5 a_j^i x_j^2, \;\;\; i=0, 1, 2,
\end{equation*}
where $a_i\neq a_j$ if $i\neq j$, by \cite[Theorem 6.1 (3)]{Knu02} we ensure that $X$ is the complete intersection of $S$ and a hypersurface of degree 2. Under the hypothesis of (iii) or (iv), when we solve the system of equations in the parameter $t$ as we previously did, we obtained that the coefficient of every mixed term vanishes and $d_5$ have the form of \eqref{d5diagonal}. The remaining conditions, (iv) or (iii) respectively, do not impose new conditions on the coefficients.   
\end{proof}

\subsection{Theta characteristics}\label{theta}
In this section we use the coverings given by the subgroups generated by involutions $\sigma_i$'s and the Rosenhain tetrahedra of singular Kummer surfaces to construct theta characteristics on a Humbert-Edge's curve of type 5. We recall the definition of a theta characteristic and a vanishing thetanull.

\begin{definition}
Let $X$ be an algebraic curve. A line bundle $L$ on $X$ is a \emph{theta characteristic} if $L^2\sim K_X$. A theta characteristic $L$ is \emph{even} (respectively, \emph{odd}) if $h^0(L)$ is even (respectively, odd). A \emph{vanishing thetanull} is an even theta characteristic $L$ such that $h^0(L)>0$.
\end{definition}
 
Recall that given a Humbert-Edge's curve $X_5$ of type 5, for every $i=0,\dots,5$ the double covering $\pi_i:X_5\rightarrow X_5/\langle \sigma_i\rangle$ is ramified at 16 points obtained as the intersection points of $X_5$ with the hyperplane $V(x_i)$. Denote by $R_i=p_{i1}+\cdots+p_{i16}$ the ramification divisor for every $i=0,\dots,5$.  

\begin{proposition}
Let $X_5\subset\mathbb{P}^5$ be a Humbert-Edge's curve of type 5. $X_5$ admits 26 odd theta characteristics with 3 sections, 6 of them corresponds to the line bundle associated with the ramification divisors and the remaining 20 are induced by the coverings associated to the subgroups generated by three different involutions $\sigma_i$, $\sigma_j$ and $\sigma_k$ for $i,j,k\in\{0,\dots,5\}$.
\end{proposition}
\begin{proof}
For $i,j\in\{0,\dots,5\}$, consider the subgroup generated by the involutions $ \sigma_i$ and $\sigma_j$ and take the induced covering of degree four $\pi_{ij}:X_5 \rightarrow X_5/\langle \sigma_i,\sigma_j\rangle$. This is a simply ramified covering with the 32 ramified points $p_{i1},\dots,p_{i16},p_{j1},\dots,p_{j16}$. Since $X_5/\langle \sigma_i,\sigma_j \rangle=E_{ij}$ is an elliptic curve, it follows that
\begin{equation*}
 K_{X_5} \sim \pi_{ij}^*(K_{E_{ij}})+R_i+R_j = R_i+R_j.
\end{equation*}
So, $K_{X_5}\sim R_i+R_j$ for all $i,j\in\{0,\dots,5\}$. Fix and index $i\in\{0,\dots,5\}$ and take $j,k\in\{0,\dots,5\}\backslash\{i\}$ with $j\neq k$. Using the fact that 
\begin{equation*}
 R_i+R_k \sim K_{X_5} \sim R_k+R_j,
\end{equation*}
we have that $R_i\sim R_j$. Thus, $K_{X_5}\sim R_i+R_j\sim 2 R_i$ and $R_i$ is a theta characteristic.

Now, let $i,j,k\in\{0,\dots,5\}$ be different indices. The covering of degree eight $\pi_{ijk}:X_5\rightarrow X_5/\langle \sigma_i,\sigma_j,\sigma_k\rangle$ is a simply ramified covering in the 48 points $p_{i1},\dots,p_{i16}, p_{j1},\dots,p_{j16}, p_{k1},\dots,p_{k16}$. Note that $X_5/\langle \sigma_i,\sigma_j,\sigma_k\rangle \cong\mathbb{P}^1$. Then,
\begin{equation*}
 K_{X_5} \sim \pi_{ijk}^*(K_{\mathbb{P}^1})+R_i+R_j+R_k\sim \pi_{ijk}^*(K_{\mathbb{P}^1})+R_i+K_{X_5}.
\end{equation*}
Therefore, we have that $\pi_{ijk}^*(-K_{\mathbb{P}^1})\sim R_i$ and we conclude that $\pi_{ijk}^*(\mathcal{O}_{\mathbb{P}^1}(2))$ is a theta characteristic of $X_5$. 

Next step is to compute $h^0(\pi_{ijk}^*(\mathcal{O}_{\mathbb{P}^1}(2))$. To do so, we will use the fact that $h^0(\pi_{ijk}^*(\mathcal{O}_{\mathbb{P}^1}(2)))=h^0({\pi_{ijk}}_*\pi_{ijk}^*(\mathcal{O}_{\mathbb{P}^1}(2)))$. The covering $\pi_{ijk}$ is determined by a line bundle $\mathcal{L}$ on $\mathbb{P}^1$ such that $\mathcal{L}^8=\mathcal{O}_{\mathbb{P}^1}({\pi_{ijk}}_*(Ri+R_j+R_k))$ and in addition, we have that ${\pi_{ijk}}_*\mathcal{O}_{X_5}=\mathcal{O}_{\mathbb{P}^1}\oplus\mathcal{L}^{-1}\oplus\cdots\oplus\mathcal{L}^{-7}$. By the projection formula:
\begin{align*}
 {\pi_{ijk}}_*\pi_{ijk}^*(\mathcal{O}_{\mathbb{P}^1}(2)) 
 & = \mathcal{O}_{\mathbb{P}^1}(2) \otimes {\pi_{ijk}}_*\mathcal{O}_{X_5} \\
 & = \mathcal{O}_{\mathbb{P}^1}(2) \otimes ( \mathcal{O}_{\mathbb{P}^1}\oplus\mathcal{L}^{-1}\oplus\cdots\oplus\mathcal{L}^{-7} ) \\
 & = (\mathcal{O}_{\mathbb{P}^1}(2) \otimes \mathcal{O}_{\mathbb{P}^1}) \oplus (\mathcal{O}_{\mathbb{P}^1}(2) \otimes \mathcal{L}^{-1} ) \oplus\cdots\oplus (\mathcal{O}_{\mathbb{P}^1}(2) \otimes\mathcal{L}^{-7}).
\end{align*}
From the equality $\mathcal{L}^8=\mathcal{O}_{\mathbb{P}^1}({\pi_{ijk}}_*(Ri+R_j+R_k))$ we get that the degree of $\mathcal{L}$ is equal to 6, this implies that $\mathcal{O}_{\mathbb{P}^1}(2) \otimes\mathcal{L}^{-n}$ has no sections for every $n=1,\dots, 7$. Therefore, ${\pi_{ijk}}_*\pi_{ijk}^*(\mathcal{O}_{\mathbb{P}^1}(2))=\mathcal{O}_{\mathbb{P}^1}(2)$ and it follows that $h^0(\pi_{ijk}^*(\mathcal{O}_{\mathbb{P}^1}(2)))=3$. Finally, the line bundle $\mathcal{O}_{X_5}(R_i)$ has 3 sections since $R_i\sim \pi_{ijk}^*(-K_{\mathbb{P}^1})$.
\end{proof} 

Using the geometry of the Kummer surface, we can determine some vanishing thetanulls on a Humbert-Edge's curve of type 5.

\begin{proposition}\label{x5theta}
Let $X_5\subset\mathbb{P}^5$ be a Humbert-Edge's curve of type 5. $X_5$ admits 80 vanishing thetanulls with 6 sections corresponding to Rosenhain tetrahedra of the associated singular Kummer surface $K$. \end{proposition}

\begin{proof}
By Proposition \ref{he5kummer} $X_5$ is contained in a Kummer surface $S$. Denote by $K$ the singular Kummer surface associated with $S$. For a Rosenhain tetrahedron on $K$, let $D$ be the associated divisor in $S$ (see Proposition \ref{rosenhain}). By \cite[Proposition 3.1]{Knu02} we have that $X_5$ and $D$ are dependent in $\mathrm{Pic}(S)$, in fact, $X_5$ appears as an element in the linear system $|2D|$. Using the fact that the canonical bundle of $S$ is trivial and that $X_5\in |2D|$, the adjunction formula implies that
 \begin{equation*}
  K_{X_5}=(K_S+X_5)|_{X_5}=(2D)|_{X_5}=2 D|_{X_5}.
 \end{equation*}
Thus, the divisor $D|_{X_5}$ is a theta characteristic. It only remains the calculation of $h^0(D|_{X_5})$. Twisting the exact sequence of sheaves
\begin{equation*}
 0 \rightarrow \mathcal{O}_S(-X_5) \rightarrow \mathcal{O}_S \rightarrow \mathcal{O}_{X_5} \rightarrow 0
\end{equation*}
we obtain the exact sequence
\begin{equation*}
 0 \rightarrow \mathcal{O}_S(-D) \rightarrow \mathcal{O}_S(D) \rightarrow \mathcal{O}_{X_5}(D) \rightarrow 0,
\end{equation*}
and therefore we obtain the exact sequence in cohomology
\begin{equation*}
 0 \rightarrow H^0(S,-D) \rightarrow H^0(S, D) \rightarrow H^0(X_5, D|{X_5}) \rightarrow H^1(S,-D).
\end{equation*}
We have $H^0(S,-D)=0$ and since $D$ is very ample, by Mumford vanishing theorem we obtain that $H^1(S,-D)=0$. Then, $H^0(S, D) = H^0(X_5, D|{X_5})$ and from $\mathrm{dim}|D|=5$ it follows that $h^0(X_5, D|{X_5})=6$. We conclude the proof recalling that there are 80 Rosenhain tetrahedra associated to a singular Kummer surface (see Proposition \ref{rosenhain}).
\end{proof}

We conclude this subsection with the following remarks:
\begin{itemize}
    \item The way to construct the vanishing thetanulls for a Humbert-Edge's curve of type 5 using the geometry of a Kummer surface differs completely to the classical case: a Humbert's curve $X$ admits exactly 10 vanishing thetanulls and can be constructed by taking the quotients of $X$ by the subgroup generated by two involutions (see the proof of \cite[Theorem 2.2]{FZ18}), the geometry of the Del Pezzo surface containing $X$ is not involved in such process.
    \item The procedure used in Proposition \ref{x5theta} to construct the vanishing thetanulls holds true for every smooth curve on degree 16 and genus 17 on a Kummer surface $S$. Indeed, if $Y$ is any smooth curve of degree 16 and genus 17 contained $S$, then using again \cite[Proposition 3.1]{Knu02} we have that $Y$ and $D$ are dependent on $\mathrm{Pic}(S)$ and the previous argument holds.
\end{itemize}

\subsection{Moduli space of Humbert-Edge's curves of type 5}\label{moduli}
As we mention in Remark \ref{x5hyp2}, given a Humbert-Edge's curve of type 5 always it is possible to associate a hyperlliptic curve of genus 2 and vice-versa. Here, we discuss about this fact and using the results of Carocca, G\'onzalez-Aguilera, Hidalgo and Rodr\'iguez \cite{CGHR08} we prove that the moduli space of Humbert-Edge's curve of type 5 is isomorphic to the moduli space of hyperelliptic curves of genus 2.

Given a Humbert-Edge's curve $X_5$ of type 5, using Edge's idea of consider the coefficients as points in $\mathbb{P}^3$ and the unique rational normal curve in $\mathbb{P}^3$ through them, one is able to write down the equations of $X_5$ in normal form (see the proof of Proposition \ref{normalx5}):
\begin{equation}
 Q_i=\sum_{j=0}^5 a_{j}^i x_j^2, \;\;\; i=0,\dots, 3,
\end{equation}
where $a_j\in\mathbb{C}$ for $j=0,\dots,5$ and $a_j\neq a_k$ if $j\neq k$. Note that since the rational normal curve that we are considering is constructed via the Veronese map $\nu:\mathbb{P}^1\rightarrow\mathbb{P}^3$, one can fix three points in $\mathbb{P}^1$ and therefore $X_5$ depends only on three different complex numbers. Thus, we can assume that we are fixing the points $0,1$ and $\infty$ and then we have three free different parameters $\lambda_1, \lambda_2$ and $\lambda_3$ defining the curve $X_5$. Since this idea can be carry out in the general case of Humbert-Edge's curves of type $n$, with this fact in mind in \cite[Section 4.1]{CGHR08} the authors found an embedding for Humbert-Edge's curves of type $n$ in $\mathbb{P}^n$ in such way that the equations depends on $n-2$ different parameters. In the particular case of the Humbert-Edge's curve $X_5$ of type 5, the equations take the form
\begin{align*}
x_0^2 + x_1^2+ x_2^2 & =0 \\
\lambda_1 x_0^2 + x_1^2 + x_3^2 & =0 \\
\lambda_2 x_0^2 + x_1^2+ x_4^2 & =0 \\
\lambda_3 x_0^2 + x_1^1 + x_5^2 &=0,
\end{align*}
where $\lambda_1,\lambda_2,\lambda_3\in\mathbb{C}\backslash\{0,1\}$ are different complex numbers. To emphasis the dependence on the parameters $\lambda_1,\lambda_2,\lambda_3$ and considering that we are fixing $0,1,\infty$, we denote this curve as $X_5(\lambda_1,\lambda_2,\lambda_3)$. Also, note that if we consider the degree 32 map given by 
\begin{align*}
\pi_{(\lambda_1,\lambda_2,\lambda_3)}:X_5(\lambda_1,\lambda_2,\lambda_3)  \rightarrow & \hspace{7mm} \mathbb{P}^1 \\
(x_0:\dots:x_5)  \mapsto & -\left(\frac{x_1}{x_0}\right)^2,
\end{align*} 
then
\begin{equation}\label{branch}
    \{0, 1,\infty, \lambda_1, \lambda_2, \lambda_3\}
\end{equation}
is the branch locus of $\pi_{(\lambda_1,\lambda_2,\lambda_3)}$.

On the other hand, if $C$ is a hyperelliptic curve of genus 2, then we can write the equation which defines $C$ as
\begin{equation}  y^2=(x-a_0)(x-a_1)\cdots(x-a_5),
\end{equation}
where $a_j\in\mathbb{C}$ for $j=0,\dots,5$ and $a_j\neq a_k$ if $j\neq k$. Since always there exists an automorphism of $\mathbb{P}^1$ which carries a tuple of different complex numbers $(a_0,a_1,a_2)$ to $(0,1,\infty)$, we may assume that $C$ is given by the equation
\begin{equation}  y^2=x(x-1)(x-\lambda_1)(x-\lambda_2)(x-\lambda_3),
\end{equation}
where $\lambda_1,\lambda_2,\lambda_3\in\mathbb{C}\backslash\{0,1\}$ are different. Similarly as before, we denoted by $C(\lambda_1,\lambda_2,\lambda_3)$ to the hyperelliptic curve of genus 2 with parameters $0,1,\infty,\lambda_1,\lambda_2$ and $\lambda_3$. In addition, since $C(\lambda_1,\lambda_2,\lambda_3)$ is a hyperelliptic curve of genus 2, there exists a degree 2 map $\rho_{(\lambda_1,\lambda_2,\lambda_3)}:C(\lambda_1,\lambda_2,\lambda_3)\rightarrow\mathbb{P}^1$ whose branch locus is precisely given by (\ref{branch}).

In both cases of Humbert-Edge's curves of type 5 and hyperelliptic curves of genus 2, given such a curve we have a map to $\mathbb{P}^1$ with a specific branch locus. In fact, the branch locus determines the curve modulo isomorphism. Indeed, Hidalgo, Reyes-Carocca and Vald\'es determined in \cite[Section 2.3]{HRV13} when two generalized Fermat curves are isomorphic, in particular, the following result can be deduced:

\begin{proposition}\label{x5iso}
Let $X_5(\lambda_1,\lambda_2,\lambda_3)$ and $X_5(\mu_1,\mu_2,\mu_3)$ be a Humbert-Edge's curves of type 5. The following statements are equivalent:
\begin{enumerate}
    \item $X_5(\lambda_1,\lambda_2,\lambda_3)$ is isomorphic to $X_5(\mu_1,\mu_2,\mu_3)$.
    \item There exists a M\"{o}bius transformation $M$ such that 
    \begin{equation*}
        \{M(0),M(1),M(\infty),M(\lambda_1),M(\lambda_2),M(\lambda_3)\}=\{0,1,\infty,\mu_1,\mu_2,\mu_3\}.
    \end{equation*}
\end{enumerate}
\end{proposition}
In the case of hyperelliptic curves of genus 2 we have an analogous result (see \cite[Section III.7.3]{FK92}):
\begin{proposition}\label{hyperiso}
Let $C(\lambda_1,\lambda_2,\lambda_3)$ and $C(\mu_1,\mu_2,\mu_3)$ be hyperelliptic curves of genus 2. The following statements are equivalent:
\begin{enumerate}
    \item $C(\lambda_1,\lambda_2,\lambda_3)$ is isomorphic to $C(\mu_1,\mu_2,\mu_3)$.
    \item There exists a M\"{o}bius transformation $M$ such that 
    \begin{equation*}
        \{M(0),M(1),M(\infty),M(\lambda_1),M(\lambda_2),M(\lambda_3)\}=\{0,1,\infty,\mu_1,\mu_2,\mu_3\}.
    \end{equation*}
\end{enumerate}
\end{proposition}

The above results can be summarized in the following commuting diagram:
\begin{equation*}
\xymatrix{
& & X_5(\lambda_1,\lambda_2,\lambda_3) \ar[dd]_(0.4){\pi_{(\lambda_1,\lambda_2,\lambda_3)}} \ar[rr]^{\cong} & & X_5(\mu_1,\mu_2,\mu_3) \ar[dd]^(0.4){\pi_{(\mu_1,\mu_2,\mu_3)}} & &  \\
& & & & & & & \\
C(\lambda_1,\lambda_2,\lambda_3) \ar[rr]^(.6){\rho_{(\lambda_1,\lambda_2,\lambda_3)}} \ar@{->}@/_{9mm}/[rrrrrr]_{\cong} & &\mathbb{P}^1 \ar[rr]^(.5){M} & & \mathbb{P}^1 \ar@{<-}[rr]^(0.4){\rho_{(\mu_1,\mu_2,\mu_3)}} & & C(\mu_1,\mu_2,\mu_3)}
\end{equation*}

We denote by $\mathcal{HE}_5$ the moduli space of Humbert-Edge's curve of type 5, that is, the set of Humbert-Edge's curves of type 5 modulo isomorphism. Similarly, we denote by $\mathcal{H}_2$ the moduli space of hyperelliptic curves of genus 2. By the above discussion we have a well-defined map
\begin{align*}
    f:\mathcal{H}_2 & \rightarrow \mathcal{HE}_5 \\
    [C(\lambda_1,\lambda_2,\lambda_3)] & \mapsto [X_5(\lambda_1,\lambda_2,\lambda_3)]
\end{align*}
By construction, it is immediate that $f$ is a surjective map. The following proposition deals with the injectivity.

\begin{proposition}\label{moduliinj}
The above map $f:\mathcal{H}_2  \rightarrow \mathcal{HE}_5$ is injective. Therefore, $f$ is an isomorphism of moduli spaces.
\end{proposition}
\begin{proof}
Let $C(\lambda_1,\lambda_2,\lambda_3)$ and $C(\mu_1,\mu_2,\mu_3)$ be hyperelliptic curves of genus 2. Assuming that $f$ is an injective map we have that $f([C(\lambda_1,\lambda_2,\lambda_3)])=f([C(\mu_1,\mu_2,\mu_3)])$, that is,  $[X_5(\lambda_1,\lambda_2,\lambda_3)]=[X_5(\mu_1,\mu_2,\mu_3)]$. Then, there exists an isomorphism between $X_5(\lambda_1,\lambda_2,\lambda_3)$ and $X_5(\mu_1,\mu_2,\mu_3)$. By Proposition \ref{x5iso} there exists a M\"{o}bius transformation $M$ such that 
\begin{equation*}
        \{M(0),M(1),M(\infty),M(\lambda_1),M(\lambda_2),M(\lambda_3)\}=\{0,1,\infty,\mu_1,\mu_2,\mu_3\}.
    \end{equation*}
Thus, by Proposition \ref{hyperiso} we conclude that the hyperelliptic curves $C(\lambda_1,\lambda_2,\lambda_3)$ and $C(\mu_1,\mu_2,\mu_3)$ are isomorphic.   
\end{proof}

Since every Humbert-Edge's curve of type 5 has genus 17, there is a natural map $q:\mathcal{HE}_5\rightarrow\mathcal{M}_{17}$ between $\mathcal{HE}_5$ and the moduli space $\mathcal{M}_{17}$ of curves of genus 17. In \cite[Proposition 4.3]{CGHR08} the authors proved that such map is injective.
Thus, using the isomorphism $f$ we obtain the following: 

\begin{corollary}
The moduli space $\mathcal{H}_2$ of hyperelliptic curves of genus 2 is a three-dimensional closed algebraic variety in $\mathcal{M}_{17}$ via the composition
$q\circ f: \mathcal{H}_2 \hookrightarrow \mathcal{M}_{17}$.
\end{corollary}

We conclude this paper noting that in the general case, the moduli space of Humbert-Edge's curves of type $n$, where $n\geq 5$ is an odd number,  is isomorphic to the moduli space of hyperelliptic curves of genus $\frac{n-1}{2}$. A Humbert-Edge's curve $X_n(\lambda_1,\dots,\lambda_{n-2})$ of type $n$ in $\mathbb{P}^n$ can be written as
\begin{align*}
x_0^2 + x_1^2+ x_2^2 & =0 \\
\lambda_1 x_0^2 + x_1^2 + x_3^2 & =0 \\
\lambda_2 x_0^2 + x_1^2+ x_4^2 & =0 \\
& \vdots\\
\lambda_{n-2} x_0^2 + x_1^1 + x_n^2 &=0,
\end{align*}
where $\lambda_1,\dots,\lambda_{n-2}\in\mathbb{C}\backslash\{0,1\}$ are different (see \cite[Section 4.1]{CGHR08}). On the other hand, a hyperelliptic curve $C(\lambda_1,\dots,\lambda_{2g-1})$ of genus $g$ is given by the equation
\begin{equation*}
    y^2=x(x-1)(x-\lambda_1)\cdots(x-\lambda_{2g-1}),
\end{equation*}
where $\lambda_1,\dots,\lambda_{2g-1}\in\mathbb{C}\backslash\{0,1\}$ are different. Since Propositions \ref{x5iso} and \ref{hyperiso} in fact hold true in the general case (see \cite[Section 2.3]{HRV13} and \cite[Section III.7.3]{FK92} respectively), under the assumptions that $n\geq 5$ is odd and $n-2=2g-1=r$, we can define a surjective map
\begin{align*}
    f_r:\mathcal{H}_g & \rightarrow \mathcal{HE}_n \\
    [C(\lambda_1,\dots,\lambda_r)] & \mapsto [X_n(\lambda_1,\dots,\lambda_r)],
\end{align*}
where $\mathcal{H}_g$ denotes the moduli space of hyperelliptic curves of genus $g=\frac{n-1}{2}$ and $\mathcal{HE}_n$ denotes the moduli space of Humbert-Edge's curve of type $n$. Finally, applying the argument of Proposition \ref{moduliinj} and using the fact that the natural map $q_n:\mathcal{HE}_n\rightarrow\mathcal{M}_{g_n}$ is injective (since any Humbert-Edge's curve of type $n$ has genus $g_n=2^{n-2}(n-3)+1$ and \cite[Proposition 4.3]{CGHR08}) we conclude the following:

\begin{proposition}
If $n\geq 5$ is an odd number and $n-2=2g-1=r$, then the map $f_r:\mathcal{H}_g  \rightarrow \mathcal{HE}_n$ is an isomorphism of moduli spaces. In particular, we have that $\mathcal{H}_g$ is an $(n-2)$-dimensional closed variety in $\mathcal{M}_{g_n}$ via the composition $q_n\circ f_r:\mathcal{H}_g\hookrightarrow\mathcal{M}_{g_n}$, where $g_n=2^{n-2}(n-3)+1$. 
\end{proposition}

\end{document}